\documentclass{article}
\usepackage[margin=1in]{geometry}
\usepackage{amsfonts, amsmath, amsthm, amssymb, bbm, enumitem, hyperref, mathrsfs, titlefoot, xcolor}

\hypersetup{
  colorlinks   = true, %Colours links instead of ugly boxes
  urlcolor     = blue, %Colour for external hyperlinks
  linkcolor    = blue, %Colour of internal links
  citecolor   = red %Colour of citations
}

\newtheorem{theorem}{Theorem}
\newtheorem{lemma}{Lemma}[section]

\theoremstyle{definition}
\newtheorem{definition}[lemma]{Definition}

\theoremstyle{remark}
\newtheorem{remark}[lemma]{Remark}

\newcommand{\lt}{\left}
\newcommand{\rt}{\right}
\newcommand{\fr}{\frac}

\newcommand{\bbE}{{\mathbb{E}}}
\newcommand{\bbN}{{\mathbb{N}}}
\newcommand{\bbP}{{\mathbb{P}}}
\newcommand{\bbR}{{\mathbb{R}}}
\newcommand{\bbS}{{\mathbb{S}}}
\newcommand{\bbZ}{{\mathbb{Z}}}

\newcommand{\cN}{{\mathcal{N}}}
\newcommand{\cK}{{\mathcal{K}}}

\newcommand{\la}{\langle}
\newcommand{\ra}{\rangle}

\newcommand{\norm}[1]{{\lt\|#1\rt\|}}

\newcommand{\diag}{{\mathrm{diag}}}
\newcommand{\tr}{{\mathsf{tr}}}
\newcommand{\TV}{{\mathsf{TV}}}
\newcommand{\KL}{{\mathsf{KL}}}
\newcommand{\Leb}{{\mathsf{Leb}}}
\newcommand{\ind}{{\mathbbm{1}}}

\newcommand{\tX}{{\widetilde X}}
\newcommand{\unif}{{\mathrm{unif}}}

\newcommand{\de}{{\mathrm{d}}}

\DeclareMathOperator*{\EE}{{\bbE}}
\DeclareMathOperator*{\PP}{{\bbP}}

\author{
	Matthew Brennan\thanks{Department of Electrical Engineering and Computer Science, Massachusetts Institute of Technology}
	\and
	Guy Bresler\thanks{\texttt{guy@mit.edu}, Department of Electrical Engineering and Computer Science, Massachusetts Institute of Technology}
	\and
	Brice Huang\thanks{\texttt{bmhuang@mit.edu}, Department of Electrical Engineering and Computer Science, Massachusetts Institute of Technology}
}
\date{June 29, 2022}
\title{Threshold for Detecting High Dimensional Geometry \\ in Anisotropic Random Geometric Graphs}

\begin{document}

\maketitle

\begin{abstract}
	In the anisotropic random geometric graph model, vertices correspond to points drawn from a high-dimensional Gaussian distribution and two vertices are connected if their distance is smaller than a specified threshold.
	We study when it is possible to hypothesis test between such a graph and an Erd\H{o}s-R\'enyi graph with the same edge probability.
	If $n$ is the number of vertices and $\alpha$ is the vector of eigenvalues, \cite{eldan2020information} shows that detection is possible when $n^3 \gg (\norm{\alpha}_2/\norm{\alpha}_3)^6$ and impossible when $n^3 \ll (\norm{\alpha}_2/\norm{\alpha}_4)^4$.
	We show detection is impossible when $n^3 \ll (\norm{\alpha}_2/\norm{\alpha}_3)^6$, closing this gap and affirmatively resolving the conjecture of \cite{eldan2020information}.
\end{abstract}

\section{Introduction}

Extracting information from large graphs is a fundamental statistical task.
Because many natural networks have underlying metric structure -- for example, nearby proteins in a biological network are more likely to share function, and users with similar interests in a social network are more likely to interact -- a central inference problem is to infer latent geometric structure in an observed graph.
Moreover, with the proliferation of large data sets in the modern world, statistical inference is inherently high dimensional, see e.g. the survey \cite{johnstone2006high}.
This motivates the study of inferring latent high dimensional geometry in a graph.

In this paper, we consider the hypothesis testing problem that determines if such inference is information-theoretically possible.
This continues a line of work originated by Bubeck, Ding, Eldan, and R\'acz \cite{bubeck2016testing} and continued by Eldan and Mikulincer \cite{eldan2020information}.
Our main contribution is a tight characterization of when detection is possible in the anisotropic setting introduced in \cite{eldan2020information}.

Formally, given a graph $G$ on $[n]$, we wish to test between two hypotheses.
The null hypothesis is that $G$ is a sample from the Erd\H{o}s-R\'enyi graph $G(n,p)$, where each edge is present with independent probability $p$.
The alternative hypothesis is that $G$ is a sample from a random geometric graph (RGG), which we define precisely below.
In such graphs, each vertex corresponds to a random point in some metric space and an edge exists between two vertices if their distance is smaller than a given threshold.

Arguably the most natural RGG is the isotropic model: each vertex $i\in [n]$ corresponds to an independent latent vector $X_i$ sampled from the Haar measure on the sphere $\bbS^{d-1}$ or an isotropic $d$-dimensional Gaussian,\footnote{These models have the same threshold for the detection task we consider.} and edge $(i,j)$ is present if $\la X_i,X_j\ra \ge t_{p,d}$, where $t_{p,d}$ is chosen so that each edge is present with probability $p$.
Let $G(n,p,d)$ denote the isotropic RGG with spherical latent data; we fix $p\in (0,1)$ and allow $d$ to vary with $n$.
The following seminal result of Bubeck, Ding, Eldan, and R\'acz characterizes, for fixed $p\in (0,1)$, when it is possible to test between $G(n,p)$ and $G(n,p,d)$.
Let $\TV$ denote total variation distance.
\begin{theorem}\cite{bubeck2016testing}
	\label{thm:bder}
	Let $p\in (0,1)$ be fixed.
	\begin{enumerate}[label=(\alph*),ref=\alph*]
		\item \label{itm:bder-lb} If $n^3 \ll d$, then $\TV(G(n,p), G(n,p,d)) \to 0$.
		\item \label{itm:bder-ub} If $n^3 \gg d$, then $\TV(G(n,p), G(n,p,d)) \to 1$.
	\end{enumerate}
\end{theorem}

Each coordinate of the latent vector $X_i$ represents an attribute of vertex $i$.
The isotropic model assumes that each attribute has the same influence on the network structure.
In real networks, some attributes are more important than others: for example, in a social network geographic location has a stronger influence on connectivity than preference of ice cream flavor.
This motivates the following anisotropic generalization of the RGG, introduced in \cite{eldan2020information}, in which attributes may have different weights.
\begin{definition}[Anisotropic random geometric graph]
	\label{defn:rgg}
	For $n\in \bbN$, $p\in (0,1)$, $\alpha \in \bbR_{\ge 0}^d$, let $G(n,p,\alpha)$ be the following random graph model.
	Generate $X_1,\ldots,X_n$ i.i.d. from $\cN(0, D_\alpha)$ where $D_\alpha$ is the diagonal matrix with diagonal $\alpha$.
	Let $t_{p,\alpha} \in \bbR$ be the unique number satisfying $\PP(\la X_1,X_2\ra \ge t_{p,\alpha}) = p$.
	Then, $G \sim G(n,p,\alpha)$ is a graph on $[n]=\{1,\ldots,n\}$ where $(i,j)$ is an edge if and only if $\la X_i,X_j\ra \ge t_{p,\alpha}$.
\end{definition}
By rotational invariance of the model, the assumption that $X_i$ has diagonal covariance is without loss of generality.
Thus, all our results apply to the case of latent Gaussian vectors with arbitrary covariance.

Throughout, we fix $p\in (0,1)$ and allow $d,\alpha$ to vary with $n$.
The central question we study is, under what limiting behaviors of $(n,d,\alpha)$ can one statistically distinguish $G(n,p,\alpha)$ from $G(n,p)$?
This question was first studied in \cite{eldan2020information}, in which the following upper and lower bounds on detection were derived.
\begin{theorem}{\cite[Theorem 2]{eldan2020information}}
	\label{thm:eldan-rgg}
	Let $p\in (0,1)$ be fixed. Then,
	\begin{enumerate}[label=(\alph*),ref=\alph*]
		\item \label{itm:rgg-lb} If $n^3 \ll (\norm{\alpha}_2/\norm{\alpha}_4)^4$, then $\TV(G(n,p), G(n,p,\alpha)) \to 0$.
		\item \label{itm:rgg-ub} If $n^3 \gg (\norm{\alpha}_2/\norm{\alpha}_3)^6$, then $\TV(G(n,p), G(n,p,\alpha)) \to 1$.
	\end{enumerate}
\end{theorem}
When $\alpha = 1^d$, this recovers the sharp $d\asymp n^3$ detection threshold in the isotropic model.
However, for general $\alpha$ there is a polynomially sized gap between the upper and lower bounds of Theorem~\ref{thm:eldan-rgg}.
For example, if $\alpha_i = i^{-1/3}$, $(\norm{\alpha}_2/\norm{\alpha}_4)^4 \asymp d^{2/3}$ while $(\norm{\alpha}_2/\norm{\alpha}_3)^6 \asymp d$.
\cite[Conjecture 1]{eldan2020information} conjectures that the hypothesis of part (\ref{itm:rgg-lb}) can be weakened to $n^3 \ll (\norm{\alpha}_2/\norm{\alpha}_3)^6$, i.e. the detection threshold is $n^3 \asymp (\norm{\alpha}_2/\norm{\alpha}_3)^6$.

The main result of this paper is to affirmatively resolve this conjecture.
\begin{theorem}[Main result]
	\label{thm:main-rgg}
	If $p\in (0,1)$ is fixed and $n^3 \ll (\norm{\alpha}_2/\norm{\alpha}_3)^6$, then $\TV(G(n,p), G(n,p,\alpha)) \to 1$.
\end{theorem}
In light of Theorem~\ref{thm:bder}, this result can be interpreted as meaning that for the task of detecting geometry, the effective dimension of the anisotropic RGG is $(\norm{\alpha}_2/\norm{\alpha}_3)^6$.

\cite[Conjecture 1]{eldan2020information} is motivated by the fact that Theorem~\ref{thm:eldan-rgg}(\ref{itm:rgg-ub}) is witnessed by the \emph{signed triangles statistic}
\[
	\theta(G) = \sum_{i<j<k} (G_{i,j}-p)(G_{i,k}-p)(G_{j,k}-p),
\]
which is also the optimal statistic witnessing Theorem~\ref{thm:bder}(\ref{itm:bder-ub}).
A heuristic reason to expect that this statistic is optimal is that it is the lowest-degree nontrivial term in the Fourier expansion of the likelihood ratio $\de G(n,p,\alpha)/\de G(n,p)$.
So, Theorem~\ref{thm:main-rgg} confirms this heuristic and the optimality of the signed triangles statistic in the anisotropic setting.

\subsection{Central Limit Theorem for Anisotropic Wishart Matrices}

Closely related to the anisotropic RGG is the following matrix of inner products generating $G(n,p,\alpha)$.
A sample of $G(n,p,\alpha)$ can be obtained by thresholding each entry of this matrix at $t_{p,\alpha}$.
\begin{definition}[Diagonal-removed anisotropic Wishart matrix]
	\label{defn:anisotropic-wishart}
	Let $W\sim W(n,\alpha)$ be the random $n\times n$ matrix generated as follows.
	Sample $X\in \bbR^{d\times n}$ with i.i.d. columns sampled from $\cN(0, D_\alpha)$, and set $W = \norm{\alpha}_2^{-1} \lt(X^\top X - \diag(X^\top X)\rt)$.
\end{definition}

For fixed $n$, if $d\to \infty$ and $\norm{\alpha}_\infty/\norm{\alpha}_2 \to 0$, by the multidimensional CLT $W(n,\alpha)$ converges to the following matrix of Gaussians.
\begin{definition}
	Let $M\sim M(n)$ be a symmetric random $n\times n$ matrix with zero diagonal and i.i.d. standard Gaussians above the diagonal.
\end{definition}
If we now allow $d, \alpha$ to vary with $n$, a natural question is, for which $(n,d,\alpha)$ can one test between $W(n,\alpha)$ and $M(n)$?
This can be regarded as the random matrix analog of the question of detecting geometry in random graphs.
Eldan and Mikulincer obtain the following detection lower bound.

\begin{theorem}{\cite[Theorem 4]{eldan2020information}}
	\label{thm:eldan-wishart}
	If $n^3 \ll (\norm{\alpha}_2/\norm{\alpha}_4)^4$, then $\KL(W(n,\alpha), M(n)) \to 0$.
\end{theorem}
Of course, by Pinsker's inequality this also implies $W(n,\alpha)$ and $M(n)$ converge in total variation.
Furthermore, the statistic $\theta(M) = \tr(M^3)$ distinguishes $W(n,\alpha)$ and $M(n)$ to total variation distance $1-o(1)$ when $n^3 \gg (\norm{\alpha}_2/\norm{\alpha}_3)^6$, which can be verified by computing the mean and variance of this statistic under the two hypotheses.

Similarly to above, these upper and lower bounds match for $\alpha = 1^d$, but in general there is a polynomially sized gap between them.
We prove the following result, which identifies the sharp threshold for this detection task by improving the lower bound in Theorem~\ref{thm:eldan-wishart}.
This result can be regarded as a tight CLT for anisotropic Wishart matrices.
\begin{theorem}
	\label{thm:main-wishart}
	If $n^3 \ll (\norm{\alpha}_2/\norm{\alpha}_3)^6$, then $\TV(W(n,\alpha), M(n)) \to 0$.
\end{theorem}

\subsection{Techniques}

Theorem~\ref{thm:main-rgg} follows from Theorem~\ref{thm:main-wishart} by the thresholding idea introduced in \cite{bubeck2016testing}.
Note that $G(n,p,\alpha)$ and $G(n,p)$ are entry-wise thresholdings of $W(n,\alpha)$ and $M(n)$.
Thus $\TV(G(n,p,\alpha),G(n,p))$ is upper bounded by $\TV(W(n,\alpha),M(n))$ plus a small error term from the difference of the thresholds.

Our main technical contributions are in the proof of Theorem~\ref{thm:main-wishart}.
We divide the entries of $\alpha$ into large coordinates $\alpha^+$ and small coordinates $\alpha^-$, each accounting for a constant fraction of its $L^2$ mass.
We note (Lemma~\ref{lem:reverse-cs}) that $(\norm{\alpha^-}_2/\norm{\alpha^-}_4)^4$ is of the same order as $(\norm{\alpha}_2/\norm{\alpha}_3)^6$, so Theorem~\ref{thm:eldan-wishart} is sufficient to show that $W(n,\alpha^-)$ converges in total variation to $M(n)$.

It remains to control the contributions of the large coordinates $\alpha^+$.
We consider a procedure (Lemma~\ref{lem:peel}) where we add the coordinates of $\alpha^+$ to $\alpha^-$ one by one.
Note that the effect of this operation on $W \sim W(n,\alpha)$ is to add an independent rank-one spike and scale down by a constant.
By a data processing argument, the increase in $\TV(W(n,\alpha),M(n))$ from one step of this procedure is bounded by $\TV(M(n,u), M(n))$, where $M(n,u)$ is a linear combination of $M(n)$ and an independent rank-one Gaussian spike (see Definition~\ref{defn:spiked-goe}).

This last quantity is bounded (Lemma~\ref{lem:spiked-goe}) using the Ingster-Suslina $\chi^2$ method, as $M(n,u)$ is a mixture of shifted Gaussian matrices parametrized by the spike.
This is done after conditioning on a high probability event under which the $\chi^2$ divergence's tails are integrable.
The resulting $\chi^2$ divergence is an expectation over two independent copies of the Gaussian spike, which is bounded by hypercontractivity estimates.

\subsection{Related Work}

There is a long history of work on low-dimensional random geometric graphs, see e.g. \cite{penrose2003random}.
The study of high-dimensional random geometric graphs began in \cite{devroye2011high}, which showed that the isotropic model $G(n,p,d)$ converges in total variation to $G(n,p)$ as $d\to\infty$ for $n$ fixed, and moreover that their clique numbers converge if $d \gg \log^3 n$.
\cite{bubeck2016testing} showed Theorem~\ref{thm:bder}, that the threshold for convergence of $G(n,p,d)$ and $G(n,p)$ with $p$ fixed is $d\asymp n^3$.
They conjectured that if $p = o(1)$, convergence occurs at smaller $d$; in particular, for $p = c/n$ they conjectured the threshold $d \asymp \log^3 n$.
\cite{brennan2020phase} proved convergence occurs when $d = \tilde\omega(n^3 p , n^{7/2}p^2)$, meaning the threshold does decrease with $p$.
Recently \cite{liu2022testing} proved that for $p=c/n$ ($c\ge 1$), convergence occurs when $d \gtrsim \log^{36}n$, resolving the conjecture of \cite{bubeck2016testing} up to polylog factors.
In a different direction, \cite{liu2021phase} obtain detection upper and lower bounds for soft random geometric graphs, wherein the inner products $\la X_i,X_j\ra$ determine the probability of edge $(i,j)$ being present.

There is also a growing literature on CLTs for random matrices.
\cite{chatterjee2007multivariate} proved a general multidimensional CLT using Stein's method.
\cite{jiang2015approximation} and \cite{bubeck2016testing} concurrently showed that $W(n,d) \triangleq W(n,1^d)$ converges in total variation to $M(n)$ if $d \gg n^3$.
\cite{bubeck2018entropic} generalized this result to arbitrary log-concave entry distributions, showing that the random matrix $W = d^{-1/2} (X^\top X - \diag(X^\top X))$, where $X\in \bbR^{d\times n}$ has i.i.d. entries from a log-concave measure, converges to $M(n)$ if $d/\log^2 d \gg n^3$.
\cite{racz2019smooth} refined the result of \cite{jiang2015approximation, bubeck2016testing} by computing the limiting value of $\TV(W(n,d),M(n))$ if $n,d\to\infty$ with $d/n^3 \to c$.
\cite{chetelat2019middle} showed a countable sequence of phase transitions for the Wishart ensemble $W(n,d)$: for each $k\in \bbN$, if $n^{k+3} \gg d^{k+1}$, they show that $W(n,d)$ converges to an explicit density $f_k$.
CLTs have been shown for Wishart tensors \cite{mikulincer2022clt} and Wishart matrices with arbitrary deleted entries \cite{brennan2021finetti}.
Finally, \cite{nourdin2021asymptotic} considers Wishart matrices $W = \sqrt{d} (d^{-1}X^\top X - I_n)$ where the columns of $X$ are drawn i.i.d. from $\cN(0,\Sigma)$ for $\Sigma \in \bbR^{d\times d}$ of the form
\[
	\Sigma_{i,j} = s(i-j),
\]
where $s : \bbZ \to \bbR$ is a covariance function with $s(0)=1$.
They show that $W$ converges in Wasserstein distance to a Gaussian matrix if $n^3 \ll d$ and $s \in \ell^{4/3}(\bbZ)$, and under various conditions if $s$ is the correlation function of a fractional Brownian noise.

\subsection{Notation and Preliminaries}

We adopt standard asymptotic notations: $f \gg g$ means that $f/g\to \infty$ and $f\gtrsim g$ means that $f \ge cg$ for an absolute constant $c$.
Throughout, $c,C > 0$ denote universal constants that may change from line to line.

We use $\TV$, $\KL$, and $\chi^2$ to denote total variation, Kullback-Leibler divergence, and chi-square divergence.
That is, for measures $\nu,\mu$ with $\nu$ absolutely continuous with respect to $\mu$,
\[
	\TV(\nu,\mu)
	=
	\fr12 \EE_{\xi \sim \mu} \lt|\fr{\de \nu}{\de \mu}(\xi) - 1\rt|,
	\qquad
	\KL(\nu,\mu)
	=
	\EE_{\xi \sim \mu} \fr{\de \nu}{\de \mu}(\xi)\log \fr{\de \nu}{\de \mu}(\xi),
	\qquad
	\chi^2(\nu,\mu)
	=
	\EE_{\xi \sim \mu} \lt(\fr{\de \nu}{\de \mu}(\xi) - 1\rt)^2.
\]
We recall that $\TV$ satisfies the triangle inequality and the data processing inequality $\TV(\cK(\nu), \cK(\mu)) \le \TV(\nu,\mu)$ for any Markov kernel $\cK$.
We also recall the Cauchy-Schwarz inequality $4\TV(\nu,\mu)^2 \le \chi^2(\nu,\mu)$.

\paragraph{Acknowledgements}
We are grateful to Dheeraj Nagaraj and Dan Mikulincer for helpful and stimulating discussions.
This work was done in part while the authors were participating in the Probability, Geometry, and Computation in High Dimensions program at the Simons Institute for the Theory of Computing in Fall 2020.

GB was supported by NSF CAREER award CCF-1940205.
BH was supported by a NSF Graduate Research Fellowship, a Siebel Scholarship, and NSF award DMS-2022448.

\section{Proof of Main Results}

For $g\in \bbR^n$, let $\Delta(g) = (gg^\top - \diag(gg^\top))$.
We introduce the following random matrix, consisting of a linear combination of a sample from $M(n)$ and a rank-one Gaussian spike (with diagonal removed).

\begin{definition}
	\label{defn:spiked-goe}
	For $u\in [0,1]$, let $M\sim M(n,u)$ be generated by
	\begin{equation}
		\label{eq:gen-spiked-goe}
		M = u \Delta(g) + \sqrt{1-u^2}M',
	\end{equation}
	where $g\sim \cN(0,I_n)$ and $M'\sim M(n)$ are independent.
\end{definition}

We defer the proof of the following lemma to Section~\ref{sec:spiked-goe}.
Using this lemma, we prove Theorems~\ref{thm:main-rgg} and \ref{thm:main-wishart}.
\begin{lemma}
	\label{lem:spiked-goe}
	We have that $\TV(M(n,u), M(n)) \lesssim u^3n^{3/2}$.
\end{lemma}

\subsection{Detection Lower Bound for Anisotropic Wishart Matrices}

We first prove Theorem~\ref{thm:main-wishart}.
Assume without loss of generality that $\alpha_1 \ge \cdots \ge \alpha_d \ge 0$.
Define $\alpha^+ = (\alpha_1,\ldots,\alpha_r)$ and $\alpha^- = (\alpha_{r+1},\ldots,\alpha_d)$, for the smallest $r$ such that $\norm{\alpha^+}_2^2 \ge \fr13 \norm{\alpha}_2^2$.
We may assume
\begin{equation}
	\label{eq:alphaplus-bd}
	\fr13 \norm{\alpha}_2^2 \le \norm{\alpha^+}_2^2 \le \fr23 \norm{\alpha}_2^2.
\end{equation}
Indeed, if $\norm{\alpha^+}_2^2 > \fr23 \norm{\alpha}_2^2$, because the $\alpha_i$ are decreasing we in fact have $r=1$ and $\alpha_1^2 > \fr23 \norm{\alpha}_2^2$.
In this case, $\alpha_1$, $\norm{\alpha}_2$, $\norm{\alpha}_3$ are within constant factors of each other.
Then the condition $n^3 \ll (\norm{\alpha}_2/\norm{\alpha}_3)^6$ is vacuous so there is nothing to prove.
Henceforth we assume \eqref{eq:alphaplus-bd}.
For $t=0,1,\ldots,r$, define
\[
	\alpha^t = (\alpha_1,\ldots,\alpha_t,\alpha_{r+1},\ldots,\alpha_d).
\]
These interpolate between $\alpha^-$ and $\alpha$ in the sense that $\alpha^0 = \alpha^-$, $\alpha^r = \alpha$.
\begin{lemma}
	\label{lem:peel}
	For each $t=1,\ldots,r$,
	\[
		\TV(W(n,\alpha^t),M(n))
		\le
		\TV(W(n,\alpha^{t-1}),M(n))
		+
		C\fr{\alpha_t^3}{\norm{\alpha}_2^3} n^{3/2}.
	\]
\end{lemma}
\begin{proof}
	Let $u_t = \alpha_t / \norm{\alpha^t}_2$.
	By the triangle inequality,
	\begin{equation}
		\label{eq:triangle}
		\TV(W(n,\alpha^t),M(n))
		\le
		\TV(W(n,\alpha^t),M(n,u_t))
		+
		\TV(M(n,u_t),M(n)).
	\end{equation}
	For $M\in \bbR^{n\times n}$, define the Markov kernel
	\[
		\cK(M) = u_t \Delta(g) + \sqrt{1-u_t^2} M
	\]
	where $g\sim \cN(0,I_n)$.
	Note that $W \sim W(n,\alpha^t)$, $M \sim M(n,u_t)$ can be generated by $W = \cK(W')$, $M = \cK(M')$ for $W' \sim W(n,\alpha^{t-1})$, $M' \sim M(n)$.
	By data processing,
	\[
		\TV(W(n,\alpha^t),M(n,u_t))
		\le
		\TV(W(n,\alpha^{t-1}),M(n)).
	\]
	The remaining term in \eqref{eq:triangle} can be bounded by Lemma~\ref{lem:spiked-goe}:
	\[
		\TV(M(n,u_t),M(n))
		\lesssim
		u_t^3 n^{3/2}
		=
		\fr{\alpha_t^3}{\norm{\alpha^t}_2^3} n^{3/2}
		\lesssim
		\fr{\alpha_t^3}{\norm{\alpha}_2^3} n^{3/2}
	\]
	where the final inequality uses \eqref{eq:alphaplus-bd}.
\end{proof}

\begin{lemma}
	\label{lem:reverse-cs}
	We have that $(\norm{\alpha^-}/\norm{\alpha^-}_4)^4 \gtrsim (\norm{\alpha}_2/\norm{\alpha}_3)^6$.
\end{lemma}
\begin{proof}
	Note that
	\[
		\norm{\alpha^+}_2^2 \norm{\alpha^-}_4^4
		=
		\sum_{i=1}^r \sum_{j=r+1}^d
		\alpha_i^2 \alpha_j^4
		\le
		\sum_{i=1}^r \sum_{j=r+1}^d
		\alpha_i^3 \alpha_j^3
		\le
		\norm{\alpha}_3^6.
	\]
	Thus, using \eqref{eq:alphaplus-bd},
	\[
		(\norm{\alpha^-}_2/\norm{\alpha^-}_4)^4
		\ge
		(\norm{\alpha^+}_2/\norm{\alpha}_3)^2
		(\norm{\alpha^-}_2/\norm{\alpha}_3)^4
		\gtrsim
		(\norm{\alpha}_2/\norm{\alpha}_3)^6.
	\]
\end{proof}

\begin{proof}[Proof of Theorem~\ref{thm:main-wishart}]
	By applying Lemma~\ref{lem:peel} repeatedly, we get
	\begin{align}
		\notag
		\TV(W(n,\alpha),M(n))
		&\le
		\TV(W(n,\alpha^-),M(n))
		+
		C \sum_{t=1}^r \fr{\alpha_t^3}{\norm{\alpha}_2^3} n^{3/2} \\
		\label{eq:wishart-final-bd}
		&\le
		\TV(W(n,\alpha^-),M(n))
		+
		C\fr{\norm{\alpha}_3^3}{\norm{\alpha}_2^3} n^{3/2}.
	\end{align}
	The hypothesis $n^3 \ll (\norm{\alpha}_2/\norm{\alpha}_3)^6$ implies the second term of \eqref{eq:wishart-final-bd} is $o(1)$.
	By Lemma~\ref{lem:reverse-cs} we further have $n^3 \ll (\norm{\alpha^-}_2/\norm{\alpha^-}_4)^4$.
	Therefore Theorem~\ref{thm:eldan-wishart} and Pinsker's inequality imply that the first term of \eqref{eq:wishart-final-bd} is $o(1)$.
	This concludes the proof.
\end{proof}

\begin{remark}
	We conjecture that Theorem~\ref{thm:main-wishart} remains true if the diagonal is not removed, i.e. if $W(n,\alpha)$ and $M(n)$ are replaced by the law $W^*(n,\alpha)$ of $W = \norm{\alpha}_2^{-1}(X^\top X - \norm{\alpha}_1 I_n)$, where $X$ is as in Definition~\ref{defn:anisotropic-wishart}, and the law $M^*(n)$ of a GOE matrix.
	With minor modifications, the proof of Lemma~\ref{lem:spiked-goe} in the next section generalizes if the diagonal is not removed, i.e. if $\Delta(g)$ and $M(n)$ are replaced by $gg^\top - I_n$ and $M^*(n)$.
	So, if Theorem~\ref{thm:eldan-wishart} holds without diagonal removal, the above proof can be easily modified to conclude Theorem~\ref{thm:main-wishart} without diagonal removal.
	The difficulty is that the entropy chain rule argument used to prove Theorem~\ref{thm:eldan-wishart} requires the diagonal to be removed.
\end{remark}

\subsection{Detection Lower Bound for Anisotropic RGGs}

The proof of Theorem~\ref{thm:main-rgg} is identical to that of Theorem~\ref{thm:eldan-rgg}(\ref{itm:rgg-lb}) (Theorem 2(b) in \cite{eldan2020information}), using Theorem~\ref{thm:main-wishart} in place of Theorem~\ref{thm:eldan-wishart}.

\begin{proof}[Proof of Theorem~\ref{thm:main-rgg}]
	Define the threshold functions $H_{p,\alpha}, K_p : \bbR \to \{0,1\}$ by
	\[
		H_{p,\alpha}(x) = \ind\{x\ge t_{p,\alpha}\},
		\qquad
		K_p(x) = \ind\{x\ge \Phi^{-1}(p)\},
	\]
	where $t_{p,\alpha}$ is defined in Definition~\ref{defn:rgg} and $\Phi(t) = \PP_{Z\sim \cN(0,1)}(Z \ge t)$ is the complement of the cdf of the standard Gaussian.
	Then, $G(n,p,\alpha)$ and $G(n,p)$ can be generated as the following entry-wise thresholdings of $W(n,\alpha)$ and $M(n)$:
	\[
		G(n,p,\alpha) = H_{p,\alpha}(W(n,\alpha)),
		\qquad
		G(n,p) = K_{p}(M(n)).
	\]
	Using the $\TV$ triangle inequality and data processing inequality,
	\begin{align}
		\notag
		\TV(G(n,p,\alpha),G(n,p))
		&\le
		\TV(H_{p,\alpha}(W(n,\alpha)), H_{p,\alpha}(M(n))) +
		\TV(H_{p,\alpha}(M(n)), K_p(M(n))) \\
		\label{eq:triangle-ineq-rgg}
		&\le
		\TV(W(n,\alpha), M(n)) +
		\TV(H_{p,\alpha}(M(n)), K_p(M(n))).
	\end{align}
	Since $n^3 \ll (\norm{\alpha}_2/\norm{\alpha}_3)^6$, Theorem~\ref{thm:main-wishart} implies that the first term of \eqref{eq:triangle-ineq-rgg} is $o(1)$.
	The second term of \eqref{eq:triangle-ineq-rgg} is $o(1)$ by \cite[Lemma 16]{eldan2020information}.
	Indeed, the proof of this lemma proceeds identically if the hypothesis $n^3 \ll (\norm{\alpha}_2/\norm{\alpha}_4)^4$ is weakened to $n^3 \ll (\norm{\alpha}_2/\norm{\alpha}_3)^6$.
\end{proof}

\section{Proof of TV Bound for Spiked Gaussian Matrix}
\label{sec:spiked-goe}

In this section we will prove Lemma~\ref{lem:spiked-goe}.
Let $M(n,u,g)$ be the random matrix $M$ generated by \eqref{eq:gen-spiked-goe} for $g \in \bbR^n$ fixed and $M' \sim M(n)$.
Thus $M(n,u)$ is a mixture of the random matrices $M(n,u,g)$ over latent randomness $g\sim \cN(0, I_n)$.

Further, for (always measurable and high probability) $S \subseteq \bbR^n$, let $\mu_S$ be the law of $g \sim \cN(0, I_n)$ conditioned on $g\in S$.
Let $M(n,u,S)$ be the law of $M$ generated by \eqref{eq:gen-spiked-goe} for $g\sim \mu_S$ and $M'\sim M(n)$.
This can be regarded as $M(n,u)$ conditioned on $g\in S$, and as a mixture of the $M(n,u,g)$ over $g\sim \mu_S$.

We begin with the following series of estimates.
Let $S\subseteq \bbR^n$ be a set we will specify later.
\begin{align}
	\notag
	\TV(M(n,u), M(n))
	&\le
	\TV(M(n,u), M(n,u,S))
	+
	\TV(M(n,u,S), M(n)) \\
	\label{eq:tv-conditioned}
	&\le
	\PP_{g\sim \cN(0,I_n)}(g\in S^c)
	+
	\TV(M(n,u,S), M(n)). \\
	\notag
	4\TV(M(n,u,S), M(n))^2
	&\le
	\chi^2(M(n,u,S), M(n)) \\
	\notag
	&=
	-1 +
	\EE_{A\sim M(n)}
	\lt(
		\EE_{g\sim \mu_S}
		\fr{\de M(n,u,g)}{\de M(n)}(A)
	\rt)^2 \\
	\label{eq:chisq-conditioned}
	&=
	-1 +
	\EE_{g,h\sim \mu_S}
	\EE_{A\sim M(n)}
	\fr{\de M(n,u,g)}{\de M(n)}(A)
	\fr{\de M(n,u,h)}{\de M(n)}(A).
\end{align}
The two estimates leading to \eqref{eq:tv-conditioned} are by the $\TV$ triangle inequality and the data processing inequality.
The estimate leading to \eqref{eq:chisq-conditioned} is by Cauchy-Schwarz.

These estimates are the starting point of the so-called truncated Ingster-Suslina $\chi^2$ method.
It is necessary to condition on an appropriate $S$ in \eqref{eq:tv-conditioned} so that the tails of the $\chi^2$ divergence \eqref{eq:chisq-conditioned} are integrable.
The following lemma evaluates the inner expectation in \eqref{eq:chisq-conditioned}.
\begin{lemma}
	\label{lem:inner-expectation}
	For $g,h\in \bbR^n$,
	\begin{align*}
		&\EE_{A\sim M(n)}
		\fr{\de M(n,u,g)}{\de M(n)}(A)
		\fr{\de M(n,u,h)}{\de M(n)}(A) \\
		&=
		(1-u^4)^{-n(n-1)/4}
		\exp\lt(
			\fr{u^2}{1-u^4}
			\sum_{1\le i<j\le n}
			g_ig_jh_ih_j
			-
			\fr{u^4}{1-u^4}
			\cdot
			\fr12
			\sum_{1\le i<j\le n}
			(g_i^2g_j^2 + h_i^2h_j^2)
		\rt).
	\end{align*}
\end{lemma}
\begin{proof}
	The densities of $M(n)$ and $M(n,u,g)$ are
	\begin{align*}
		\fr{\de M(n)}{\de \Leb}(A)
		&=
		\prod_{1\le i<j\le n}
		(2\pi)^{-1/2} \exp\lt(-\fr12 A_{i,j}^2\rt), \\
		\fr{\de M(n,u,g)}{\de \Leb}(A)
		&=
		\prod_{1\le i<j\le n}
		(2\pi \cdot (1-u^2))^{-1/2} \exp\lt(-\fr{1}{2(1-u^2)} (A_{i,j} - ug_ig_j)^2\rt).
	\end{align*}
	Thus
	\begin{align*}
		&
		\EE_{A\sim M(n)}
		\fr{\de M(n,u,g)}{\de M(n)}(A)
		\fr{\de M(n,u,h)}{\de M(n)}(A) \\
		&=
		(1-u^2)^{-n(n-1)/2}
		\prod_{1\le i<j\le n}
		\EE_{A_{i,j}\sim \cN(0,1)}
		\exp\lt(
			-\fr{(A_{i,j} - ug_ig_j)^2}{2(1-u^2)}
			-\fr{(A_{i,j} - uh_ih_j)^2}{2(1-u^2)}
			+A_{i,j}^2
		\rt).
	\end{align*}
	By a straightforward calculation the inner expectation equals
	\[
		\sqrt{\fr{1+u^2}{1-u^2}}
		\exp\lt(\fr{u^2}{1-u^4} g_ig_jh_ih_j - \fr{u^4}{2(1-u^4)} (g_i^2g_j^2 + h_i^2h_j^2)\rt)
	\]
	from which the result follows.
\end{proof}
Before proceeding further, we record two consequences of Boolean and Gaussian hypercontractivity.
The proof of the following lemma is standard, see \cite[Chapters 9 and 11]{o2014analysis}.
\begin{lemma}
	\label{lem:hypercontractivity}
	Let $f : \{-1,1\}^n \to \bbR$ be a polynomial of degree $d\ge 2$ and $\nu = \unif(\{-1,1\}^n)$.
	Further, let $\sigma^2 = \EE_{x \sim \nu} f(x)^2$.
	\begin{enumerate}[label=(\alph*), ref=\alph*]
		\item \label{itm:hypercontractivity-moments}
		For any $k\ge 2$, $\EE_{x \sim \nu} f(x)^k \le d^{dk/2} \sigma^k$.
		\item \label{itm:hypercontractivity-tails}
		There exist constants $C_d, c_d$ such that $\PP_{x \sim \nu} [|f(x)| \ge t\sigma] \le C_d\exp(-c_d t^{2/d})$.
	\end{enumerate}
	The same statements hold if $f : \bbR^n \to \bbR$ and $\nu = \cN(0, I_n)$.
\end{lemma}
% \begin{proof}
% 	In the Boolean case, the two parts follow from \cite[Theorem 9.21, Theorem 9.23]{o2014analysis}.
% 	In the Gaussian case, we can approximate each $x_i$ by $x_i = \fr{1}{\sqrt{N}} \sum_{j=1}^N y_{i,j}$ for i.i.d. $y_{i,j} \sim \unif(\{-1,1\})$, then apply the Boolean case with variables $(y_{i,j})_{i\le n, j\le N}$ and the Central Limit Theorem as $N\to\infty$.
% \end{proof}
For $a\ge 1$, let
\[
	S(a) = \lt\{
		g \in \bbR^n :
		\norm{g}_2^2 \le (1+a)n, \norm{g}_4^4 \le 3(1+a) n
	\rt\}.
\]
We will prove Lemma~\ref{lem:spiked-goe} by taking $S=S(a)$ in the calculation \eqref{eq:chisq-conditioned} for appropriate $a$.
We first estimate the probability of $S(a)$.
\begin{lemma}
	\label{lem:sa-prob}
	For $g \sim \cN(0,I_n)$, $\PP(g\in S(a)^c) \le C\exp(-ca^{1/2}n^{1/4})$.
\end{lemma}
\begin{proof}
	Let $f_2 = \norm{g}_2^2 - n$ and $f_4 = \norm{g}_4^4 - 3n$.
	Note that $\EE f_2^2 = 2n$ and $\EE f_4^2 = 100n$.
	By Lemma~\ref{lem:hypercontractivity}(\ref{itm:hypercontractivity-tails})
	\[
		\PP(\norm{g}_2^2 > (1+a)n)
		=
		\PP(f_2 > an)
		\le
		C_2 \exp(-c_2 an/\sqrt{2n})
		\le
		C \exp(-can^{1/2}).
	\]
	Similarly
	\[
		\PP(\norm{g}_4^4 > 3(1+a)n)
		=
		\PP(f_4 > 3an)
		\le
		C_2 \exp\lt(-c_2 (3an/\sqrt{100n})^{1/2}\rt)
		\le
		C \exp(-ca^{1/2}n^{1/4}).
	\]
\end{proof}
In light of Lemma~\ref{lem:inner-expectation}, define the random variables
\[
	X =
	\sum_{1\le i<j\le n}
	g_ig_jh_ih_j,
	\qquad
	Y =
	\fr12
	\sum_{1\le i<j\le n}
	(g_i^2g_j^2 + h_i^2h_j^2).
\]
The following two lemmas bound, respectively, the low and high moments of $X$ and $Y$ under $g,h\sim \mu_{S(a)}$.
\begin{lemma}
	\label{lem:low-moments}
	The following estimates hold for all $a\ge 1$.
	\begin{align*}
		\EE_{g,h\sim \mu_{S(a)}} X &= 0, \\
		\EE_{g,h\sim \mu_{S(a)}} XY &= 0, \\
		\lt| \EE_{g,h\sim \mu_{S(a)}} X^2 - \fr{n(n-1)}{2} \rt|
		&\lesssim \PP(S(a)^c)^{1/2} n^2, \\
		\lt| \EE_{g,h\sim \mu_{S(a)}} Y - \fr{n(n-1)}{2} \rt|
		&\lesssim \PP(S(a)^c)^{1/2} n^2, \\
		\lt| \EE_{g,h\sim \mu_{S(a)}} X^3 - n(n-1)(n-2) \rt|
		&\lesssim \PP(S(a)^c)^{1/2} n^3.
	\end{align*}
\end{lemma}
\begin{proof}
	The first two claims follow by the symmetry of $S(a)$ under the map $(g_1,\ldots,g_n) \mapsto (x_1g_1,\ldots,x_ng_n)$ for any $x \in \{-1,1\}^n$.
	In the rest of this proof, let $\EE$ denote expectation with respect to $g,h\sim \cN(0, I_n)$.
	By straightforward calculation,
	\begin{align*}
		\EE X^2 &= \fr{n(n-1)}{2}, \\
		\EE Y &= \fr{n(n-1)}{2}, \\
		\EE X^3 &= n(n-1)(n-2).
	\end{align*}
	We estimate the discrepancy caused by changing the measure from $\mu_{S(a)}$ to $\cN(0, I_n)$ using the following generic bound.
	For any $(g,h)$-measurable $\xi$,
	\[
		\EE_{g,h\sim \mu_{S(a)}}
		\xi
		=
		\PP(S(a))^{-2}
		\EE
		\ind\{g,h\in S(a)\}
		\xi
		=
		\PP(S(a))^{-2}
		\lt[
			\EE
			\xi
			-
			\EE
			\ind\{(g,h\in S(a))^c\}
			\xi
		\rt].
	\]
	Thus
	\begin{align*}
		\lt|\EE_{g,h\sim \mu_{S(a)}}\xi - \EE \xi\rt|
		&\le
		\lt(\PP(S(a))^{-2}-1\rt)
		\lt|\EE \xi \rt|
		+
		\PP(S(a))^{-2}
		\lt|
			\EE
			\ind\{(g,h\in S(a))^c\}
			\xi
		\rt| \\
		&\le
		\fr{
			2\PP(S(a)^c)
			+
			\sqrt{2\PP(S(a)^c)}
		}{\PP(S(a))^2} \sqrt{\EE(\xi^2)}
		\lesssim
		\sqrt{\PP(S(a)^c)\EE(\xi^2)}.
	\end{align*}
	For $\xi = X^2$, by Lemma~\ref{lem:hypercontractivity}(\ref{itm:hypercontractivity-moments})
	\[
		\EE \xi^2 = \EE X^4 \le 4^{8} (\EE X^2)^2 \lesssim n^4,
	\]
	which proves the third conclusion.
	For $\xi = X^3$, we similarly have $\EE \xi^2 \lesssim n^6$, proving the fifth conclusion.
	For $\xi = Y$, a straightforward calculation shows $\EE \xi^2 \lesssim n^4$, proving the fourth conclusion.
\end{proof}

\begin{lemma}
	\label{lem:high-moments}
	For $a\ge 1$ and integer $i,j\ge 0$,
	\[
		\EE_{g,h\sim \mu_{S(a)}}
		|X^iY^j|
		\le
		(6i)^i (2an)^{i+2j}.
	\]
\end{lemma}
\begin{proof}
	By Cauchy-Schwarz,
	\begin{equation}
		\label{eq:tail-terms-bd}
		\EE_{g,h\sim \mu_{S(a)}}
		|X^iY^j|
		\le
		\lt(\EE_{g,h\sim \mu_{S(a)}}X^{2i}\rt)^{1/2}
		\lt(\EE_{g,h\sim \mu_{S(a)}}Y^{2j}\rt)^{1/2}.
	\end{equation}
	For all $g \in S(a)$,
	\[
		\sum_{1\le i<j\le n}
		g_i^2g_j^2
		\le
		\norm{g_i}_2^4
		\le
		(1+a)^2n^2
		\le
		(2an)^2
	\]
	and similarly for $h$. Thus if $g,h\in S(a)$, then $|Y| \le (2an)^2$, which implies
	\begin{equation}
		\label{eq:y-big-power}
		\lt(\EE_{g,h\sim \mu_{S(a)}}Y^{2j}\rt)^{1/2}
		\le
		(2an)^{2j}.
	\end{equation}
	If $i=0$, this implies the result.
	Otherwise assume $i\ge 1$.
	By symmetry of the set $S(a)$, the distribution of $X$ under $g,h\sim \mu_{S(a)}$ is the same as that of
	\[
		\tX = \sum_{1\le i<j\le n} x_i x_j g_ig_j h_ih_j
	\]
	where $x \sim \unif(\{-1,1\}^n)$ and $g,h\sim \mu_{S(a)}$ are independent.
	By Lemma~\ref{lem:hypercontractivity}(\ref{itm:hypercontractivity-moments}), conditioned on $g,h$,
	\[
		\EE_{x}
		\tX^{2i}
		\le
		(2i)^{2i}
		\lt(\EE_{x} \tX^2 \rt)^i
		=
		(2i)^{2i}
		\lt(
			\sum_{1\le i<j\le n}
			g_i^2g_j^2 h_i^2h_j^2
		\rt)^i.
	\]
	For $g,h\in S(a)$,
	\[
		\sum_{1\le i<j\le n}
		g_i^2g_j^2 h_i^2h_j^2
		\le
		\fr12
		\sum_{1\le i<j\le n}
		(g_i^4g_j^4 + h_i^4h_j^4)
		\le
		\fr12 (\norm{g}_4^8 + \norm{h}_4^8)
		\le
		3^2(1+a)^2n^2
		\le
		3^2(2an)^2.
	\]
	So,
	\[
		\EE_{g,h\sim \mu_{S(a)}}
		X^{2i}
		=
		\EE_{g,h\sim \mu_{S(a)}}
		\EE_{x}
		\tX^{2i}
		\le
		\lt(6i\cdot 2an\rt)^{2i}.
	\]
	Recalling \eqref{eq:tail-terms-bd} and \eqref{eq:y-big-power} this implies the result.
\end{proof}

\begin{proof}[Proof of Lemma~\ref{lem:spiked-goe}]
	We may assume $u^2n \le 10^{-4}$ because otherwise the lemma is trivial.
	Take $S=S(a)$ for $a = (u^2n)^{-1/4}$.
	By Lemma~\ref{lem:sa-prob}, $\PP(S^c) \le C\exp(-cu^{-1/4}n^{1/8})$.
	Equation \eqref{eq:chisq-conditioned} and Lemma~\ref{lem:inner-expectation} imply
	\begin{align*}
		1 + 4\TV(M(n,u,S),M(n))^2
		&\le
		(1-u^4)^{-n(n-1)/4}
		\EE_{g,h\sim \mu_S}
		\exp\lt(
			\fr{u^2}{1-u^4}X
			- \fr{u^4}{1-u^4}Y
		\rt) \\
		&=
		(1-u^4)^{-n(n-1)/4}
		\lt(
			\sum_{i,j\ge 0}
			\fr{(-1)^j}{i! j!}
			\cdot
			\fr{u^{2i+4j}}{(1-u^4)^{i+j}}
			\EE_{g,h\sim \mu_S}
			X^iY^j
		\rt).
	\end{align*}
	Let $T = \{(i,j) \in \bbZ_{\ge 0}^2 : i+2j < 4\}$ and $T^c = \bbZ_{\ge 0}^2 \setminus T$.
	By the estimates in Lemma~\ref{lem:low-moments},
	\begin{align*}
		&\sum_{(i,j) \in T}
		\fr{1}{i! j!}
		\fr{u^{2i+4j}}{(1-u^4)^{i+j}}
		\EE_{g,h\sim \mu_S}
		X^iY^j \\
		&\le
		1
		- \fr{u^4}{1-u^4}    \cdot \fr{n(n-1)}{2}
		+ \fr{u^4}{(1-u^4)^2}\cdot \fr{n(n-1)}{4}
		+ \fr{u^6}{(1-u^4)^3}\cdot \fr{n(n-1)(n-2)}{6}
		+ Cn^3e^{-cu^{-1/4}n^{1/8}} \\
		&\le
		1
		- u^4 \cdot \fr{n(n-1)}{4}
		+ Cu^6n^3.
	\end{align*}
	In the last line we have used that $n^3e^{-cu^{-1/4}n^{1/8}} \ll u^{6}n^3$, because $e^{cu^{-1/4}n^{1/8}}$ is larger than any polynomial in $u^{-1}$.
	By Lemma~\ref{lem:high-moments},
	\begin{align*}
		\sum_{(i,j) \in T^c}
		\fr{1}{i! j!}
		\fr{u^{2i+4j}}{(1-u^4)^{i+j}}
		\EE_{g,h\sim \mu_S}
		X^iY^j
		&\le
		\sum_{(i,j) \in T^c}
		\fr{1}{(i/e)^i}
		\fr{u^{2i+4j}}{(1-u^4)^{i+2j}}
		\EE_{g,h\sim \mu_S}
		|X^iY^j| \\
		&\le
		\sum_{(i,j) \in T^c}
		\fr{(6i)^i}{(i/e)^i}
		\lt(\fr{2au^2n}{1-u^4}\rt)^{i+2j} \\
		&\le
		\sum_{(i,j) \in T^c}
		\lt(\fr{12e\cdot au^2n}{1-u^4}\rt)^{i+2j}.
	\end{align*}
	Since $au^2n = (u^2n)^{3/4} \le 10^{-3}$, this is a convergent double-geometric sum.
	As $i+2j \ge 4$ for $(i,j)\in T^c$,
	\[
		\sum_{(i,j) \in T^c}
		\fr{1}{i! j!}
		\fr{u^{2i+4j}}{(1-u^4)^{i+j}}
		\EE_{g,h\sim \mu_S}
		X^iY^j
		\lesssim
		(au^2n)^4
		=
		u^6n^3.
	\]
	Combining the above,
	\begin{align*}
		\log \lt(1 + 4\TV(M(n,u,S), M(n))^2\rt)
		&\le
		-\fr{n(n-1)}{4}\log(1-u^4) +
		\log \lt(
			1 - u^4 \cdot \fr{n(n-1)}{4} + Cu^6 n^3
		\rt) \\
		&\le
		Cu^6n^3.
	\end{align*}
	Therefore
	\[
		\TV(M(n,u,S), M(n))
		\lesssim
		u^3n^{3/2}.
	\]
	Finally, since $\PP(S^c) \le C\exp(-cu^{-1/4}n^{1/8}) \ll u^3n^{3/2}$, \eqref{eq:tv-conditioned} implies $\TV(M(n,u), M(n)) \lesssim u^3n^{3/2}$.
\end{proof}

\bibliographystyle{alpha}
\bibliography{bib}

\end{document}